\documentclass[a4paper,12pt]{amsart}
\usepackage[utf8]{inputenc}
\usepackage[T1]{fontenc}
\usepackage[english]{babel}
\usepackage{amsthm,amsmath,amscd,amssymb,latexsym,bbm,amsfonts}
\usepackage[dvips]{graphicx}
\usepackage{verbatim}

\theoremstyle{definition}
\newtheorem{definition}{Definition}[section]

\theoremstyle{plain}
\newtheorem{lemma}[definition]{Lemma}
\newtheorem{theorem}[definition]{Theorem}

\newtheorem{proposition}[definition]{Proposition}

\newcommand*{\bR}{\ensuremath{\mathbb{R}}}
\newcommand*{\loc}{\mathrm{loc}}
\newcommand*{\rk}{\mathrm{rank}\;\!\!}
\newcommand*{\bN}{\ensuremath{\mathbb{N}}}

\newcommand*{\spa}{\mathrm{span}}

\newcommand{\G}{\mathbb{G}}

\newcommand{\N}{\mathbb{N}}

\newcommand{\R}{\mathbb{R}}

\newcommand{\cH}{\mathcal{H}}

\newcommand{\cg}{\mathfrak{g}}

\newcommand{\res}{\mbox{\LARGE{$\llcorner$}}}

\newcommand{\der}{\partial}

\newcommand{\lls}{\big(}
\newcommand{\rrs}{\big)}

\newcommand{\ds}{\displaystyle}
\newcommand{\beqas}{\begin{eqnarray*}}
\newcommand{\eeqas}{\end{eqnarray*}}
\newcommand{\beqa}{\begin{eqnarray}}
\newcommand{\eeqa}{\end{eqnarray}}
\newcommand{\beq}{\begin{equation}}
\newcommand{\eeq}{\end{equation}}
\newcommand{\zr}{{(z)}}

\def\Xint#1{\mathchoice
   {\XXint\displaystyle\textstyle{#1}}%
   {\XXint\textstyle\scriptstyle{#1}}%
   {\XXint\scriptstyle\scriptscriptstyle{#1}}%
   {\XXint\scriptscriptstyle\scriptscriptstyle{#1}}%
   \!\int}
\def\XXint#1#2#3{{\setbox0=\hbox{$#1{#2#3}{\int}$}
     \vcenter{\hbox{$#2#3$}}\kern-.5\wd0}}
\newcommand{\meanint}{\Xint-}

\title[A Gromov's dimension comparison estimate]{A Gromov's dimension comparison estimate for rectifiable sets}
\author{Valentino Magnani}
\address{Valentino Magnani,  Dipartimento di Matematica \\
Largo Pontecorvo 5 \\ I-56127, Pisa}
\email{magnani@dm.unipi.it}

\author{Aleksandra Zapadinskaya}
\address{Aleksandra Zapadinskaya,  Dipartimento di Matematica \\
Largo Pontecorvo 5 \\ I-56127, Pisa}
\email{azapadinskaya@mail.dm.unipi.it}

\begin{document}

\thanks{The first author acknowledges the support of the European Project ERC AdG *GeMeThNES*.}
\thanks{The second author acknowledges the support of the Galileo Galilei grant, of the University of Pisa.}
\date{\today}

\begin{abstract}
We extend the validity of a Gromov's dimension comparison estimate for topological hypersurfaces to sufficiently large classes of rectifiable sets,
arising from Sobolev mappings. Our tools are a suitably weak exterior differentiation for pullback differential forms and a new low rank 
property for Sobolev mappings.
\end{abstract}

\maketitle

\tableofcontents

\section{Introduction}



The present work deals with the problem of finding the minimal regularity that still implies a certain ``transversality'' for a hypersurface with respect to a Lie bracket generating smooth distribution of linear subspaces. Clearly, the more regularity decreases, the more our ``surface'' can ``twist'', to become tangent to this distribution.
However, for very low regular surfaces tangency may have no meaning. 

This is precisely the case of a set $T$ of topological dimension $(q-1)$, contained in a Carnot-Carath\'eodory space of topological dimension $q$.
For this set, Gromov showed the following dimension comparison estimate 
\beq\label{GromovD}
\dim_HT\ge Q-1\,,
\eeq
where $Q$ is the Hausdorff dimension of the Carnot-Carath\'eodory space with respect to the sub-Riemannian distance, see 2.1 of \cite{gromov}, and
$\dim_H$ denotes the Hausdorff dimension with respect to this distance. In other words, the previous transversality problem is rephrased using the sub-Riemannian distance that arises from the smooth distribution. In the terminology of \cite{gromov}, this is the so-called {\em horizontal distribution}.

It is rather natural to ask which kind of regularity is necessary for a set $T$ in order to satisfy \eqref{GromovD}. 
For a smooth hypersurface the validity of this dimension comparison estimate follows from the transversality of the tangent space with respect to the
horizontal distribution, hence we will consider sets having tangent spaces, at least almost everywhere. 
This leads us to the general problem of testing the validity of \eqref{GromovD} for different classes of $(q-1)$-rectifiable sets, \cite{Federer69}.

First answers to this issue are contained in the work of Balogh and Tyson, who constructed a horizontal fractal in the first Heisenberg group, whose 2-dimensional Hausdorff measure with respect to the sub-Riemannian distance is finite and positive. This fractal also contains the graph of a BV function, \cite{BHT}, \cite{BT}. As a consequence, the Heisenberg group, of Hausdorff dimension $Q=4$, has a 2-rectifiable set $S$ contained in the graph of this BV function, such that 
\beq\label{dimEH}
0<\cH_d^2(S)<+\infty \quad \mbox{and }\quad \dim_H(S)= 2\,
\eeq
where $\cH^2_d$ is the Hausdorff measure with respect to the sub-Riemannian distance $d$. The estimates \eqref{dimEH} also imply that the ``approximate'' tangent space of $S$ is tangent to the horizontal distribution of the Heisenberg group, see Theorem~6.2 of \cite{MMM}. Since this group is also a Carnot-Carath\'eodory space, conditions \eqref{dimEH} show that Gromov's dimension comparison estimate \eqref{GromovD} cannot extend to all possible $(q-1)$-rectifiable sets. On the other hand, in all Heisenberg groups, each one codimensional rectifiable set of $W^{1,1}$ Sobolev regularity satisfies \eqref{GromovD}. In fact, a more general result can be proved for this group, see \cite{BHW}, \cite{MMM}.

In the present work, we consider $(q-1)$-rectifiable sets in homogeneous stratified groups and we show that under a suitable Sobolev regularity,
they must satisfy the Gromov's dimension comparison estimate \eqref{GromovD}.
\begin{theorem}[Dimension comparison for Sobolev rectifiable sets]\label{HDim}
Let $\Omega\subset\R^{q-1}$ be an open set, let $f\in W_\loc^{1,p}(\Omega,\R^q)$
be such that a.e. in $\Omega$ the rank of its differential is $q-1$.
We equip $\R^q$ with the structure of homogeneous stratified group, with homogeneous distance $d$.
We assume that $p>q-m_1$ if $m_1<q-1$, or $p=1$ if $m_1=q-1$.
If we set $\Sigma=f(\Omega)$ and $Q$ to be the Hausdorff dimension of the group with respect to $d$, then we have $\cH_d^{Q-1}(\Sigma)>0$.
In particular, we have $\dim_H\Sigma\ge Q-1$.
\end{theorem}
In this theorem, $m_1$ denotes the dimension of all horizontal fibers \eqref{eq:Hj}. Section~\ref{HypSect} provides more details on the standard identification 
of a stratified group with $\R^q$, through a graded basis.
We notice that Theorem~\ref{HDim} holds with the minimal Sobolev regularity $W^{1,1}_\loc$, whenever the horizontal distribution has codimension one.

In the Gromov's proof of \eqref{GromovD}, the fact that the set of topological dimension $q-1$ separates the space locally into two parts
plays a key role. This may be interpreted as the fact that images of Sobolev mappings satisfying the assumptions of Theorem~\ref{HDim} have somehow
a ``local separating property''. 

The proof of Theorem~\ref{HDim} relies on two independent results. The first one is a weak exterior differentiation for pullback Sobolev differential forms.

\begin{theorem}[Exterior differentiation]\label{th1}
Let $k,n,m$ be three positive integers such that $k<n$ and $k\le m$. Assume that one of the following conditions holds:
$p>k$ if $k>1$, or $p=1$ if $k=1$.
Let $\Omega\subset\bR^n$ be an open set, $f\in W^{1,p}_\loc(\Omega;\bR^m)$ and $\eta$ be a continuously differentiable $k$-form in $\bR^m$.
Then the condition $f^*\eta=0$ almost everywhere in $\Omega$ implies that $f^*(d\eta)=0$ almost everywhere in $\Omega$.
\end{theorem}
The proof of this result develops the blow-up arguments used in \cite{MMM}, with some additional difficulties. 
The important case is when $k=n-1$, where a new tool to remark is the Sobolev imbedding theorem on $(n-1)$-spheres.
This allows us to find a suitable blow-up sequence of the mapping $f$ and to perform the oriented integration of the rescaled mapping to pass to the limit. In fact, under suitable Sobolev regularity, we can introduce well defined oriented integrals of Sobolev differential forms on spheres, see Section~\ref{OInt}.
It is important to stress that in Theorem~\ref{th1} the coefficients of the Sobolev differential form $f^*\eta$ may not be even locally summable, hence its distributional exterior differentiation would not be possible. The case $k<n-1$ is a standard consequence of the previous case, as explained in Section~\ref{Sect:Slice}.

The second result for the proof of Theorem~\ref{HDim} is more delicate and it represents the main novelty of this work.
Let us start from our Sobolev rectifiable set, given by the image of a Sobolev mapping $f\in W^{1,p}(\Omega,\R^q)$ whose approximate differential $df$ has a.e.\ maximal rank, where $\Omega\subset\R^{q-1}$ is an open set. Notice that throughout our work we will not use a special notation to denote the approximate differential.
From standard arguments, one can check that the dimensional estimate \eqref{GromovD} is satisfied as soon as the image of the approximate differential of $f$ does not contain the horizontal fiber on a set of positive measure.
Thus, our starting point is to assume, by contradiction, that this image, representing the approximate tangent space, contains a.e. the horizontal fiber.
Since we wish to prove that this horizontality implies that the rank of $Df$ cannot be maximal, our horizontality assumption is formulated by two possible conditions, either
\begin{equation}\label{eq:hor12}
(H_1)_{f(y)}\subset df_y(\bR^{q-1}_y) \quad \mbox{or}\quad  df_y(\bR^{q-1}_y)\subset (H_1)_{f(y)},
\end{equation}
depending on whether the horizontal distribution of the group has codimension one.
We have denoted by $(H_1)_x$ the horizontal fiber defined in \eqref{eq:Hj}. 
This first horizontality condition of \eqref{eq:hor12} presents new difficulties with respect to the case of  the Heisenberg group, that is included in the second one.
In fact, letting $\eta_1,\ldots,\eta_q$ be the dual basis of the left invariant forms, with respect to a basis of the Lie algebra, the first condition of \eqref{eq:hor12}
does not imply any vanishing of the single pullback form 
\[
f^*\eta_i
\]
with $i>m_1$, where $m_1$ is the dimension of the horizontal fibers. For this reason, we have to use higher dimensional differential forms, getting the following vanishing condition
\[
f^*(\eta_{m_1+1}\wedge\cdots\wedge\eta_q)=0.
\]
In view of Theorem~\ref{th1}, our Sobolev regularity allows us to differentiate the previous equality, that holds almost everywhere, hence obtaining new vanishing conditions, that in turn imply that also the second layer $(H_2)_{f(y)}$ is contained in the image of $df_y(\R^{n-1})$ for a.e.\ $y\in\Omega$. For groups of step two, this leads to a contradiction.
For higher step groups, our Sobolev regularity does not allow us to perform further differentiations. Here a tricky argument by induction over the number of strata $(H_j)_{f(y)}$ contained in the approximate tangent space at $f(y)$ overcomes this problem, hence performing exterior differentiation only once.
The previous arguments are contained in the proof of Theorem~\ref{main}, that can be also seen as a new low rank property,
in the terminology of \cite{MMM}. We also point out that this theorem contains a stronger result for two step stratified groups,
that is the rank of $df$ is not maximal almost everywhere. On the other hand, in any stratified group Theorem~\ref{main} immediately implies the following result.

\begin{theorem}\label{main-I}
Let $\R^q$ be equipped with the structure of homogeneous stratified group and let $m_1$ be the dimension of horizontal fibers.
Fix  $p>q-m_1$ if $m_1<q-1$ or $p=1$ if $q=m_1+1$. 
Suppose that $f\in W_\loc^{1,p}(\Omega,\R^q)$, where $\Omega\subset\bR^{q-1}$ is open and the rank of $df$ is $q-1$ a.e.\ in $\Omega$. Then there exists a set of positive measure $A\subset\Omega$, such that $(H_1)_{f(y)}\nsubseteq df_y(\bR^{q-1}_y)$ for every $y\in A$.
\end{theorem}

This is also a consequence of the fact that in the case $m_1=q-1$, the hypothesis that $df$ has rank $q-1$ a.e.\ makes the inclusion $df_y(\R^{q-1}_y)\subset (H_1)_{f(y)}$ equivalent to the opposite inclusion $(H_1)_{f(y)}\subset df_y(\R^{q-1}_y)$. Theorem~\ref{main-I} joined with standard arguments, see for instance the proof of Theorem~1.2 in \cite{MMM}, immediately lead us to Theorem~\ref{HDim}.

\section{Oriented integral of Sobolev differential forms on spheres}\label{OInt}

Let $k,n,m$ be three positive integers such that $k\le\min\{m,n\}$. We denote by $I_{k,m}$ the set of all ordered collections $(i_1,\ldots,i_k)\in\N^k$, such that
$ 1\leq i_1<\ldots<i_k\leq m$.
The standard basis of elementary $k$-forms in $\R^m$ is given by the elements
\[
dx_J=dx_{i_1}\wedge\ldots\wedge dx_{i_k}
\]
where $J=(i_1,\ldots,i_k)$ varies on $I_{k,m}$. Let $\Omega\subset\R^n$ be an open set and assume that $f:\Omega\to\R^m$ has approximate differential $df_y$ at $y\in\Omega$ and $J=(i_1,\ldots,i_k)\in I_{k,m}$. Then we define
\[
(f^*dx_J)_y=\left(df_J\right)_y=\left(df_{i_1}\right)_y\wedge\ldots\wedge \left(df_{i_k}\right)_y,
\]
so that whenever $f$ has a.e.\ approximate differential this definition sets a $k$-form defined a.e.\ in $\Omega$, that we denote by $df_J$.
It is important to stress that even in the case $f$ is in some Sobolev space, the elementary $k$-form $df_J$ need not be even locally summable.

Minors are denoted as follows. If $J=(i_1,\ldots,i_k)\in I_{k,m}$ and $I=(j_1,\ldots,j_k)\in I_{k,n}$, then we set
\begin{equation*}
\frac{\partial f_{J}}{\partial x_{I}}(y)=\frac{\partial(f_{i_1},\ldots,f_{i_k})}{\partial(x_{j_1},\ldots,x_{j_k})}(y)=
\det\left[\frac{\partial f_{i_l}}{\partial x_{j_s}}(y)\right]_{l,s=1}^{k}\,,
\end{equation*}
where $y$ is a point of approximate differentiability of $f$. Thus, for a Sobolev mapping the functions
$\frac{\partial f_{J}}{\partial x_{I}}\colon\Omega\to\bR$ are a.e.\ well defined.

For $x\in\R^n$ and $r>0$, we define the Euclidean ball $B(x,r)=\{y\in\R^n: |x-y|<r\}$.
We will use the notation $\mathbb B$ for the unit ball $B(0,1)$ in $\bR^n$.
The notion of integration of $n-1$-form over $n-1$-manifolds in $\bR^n$ can be extended to Sobolev mappings. For our purposes, it is enough to consider integration on spheres.

Let $n\geq3$ be an integer. Let $m\geq n-1$ and $f\in W^{1,1}_\loc(\Omega;\bR^m)$.
Let $z\in\Omega$ be fixed. Then for $\cH^1$-a.e.\ $r>0$ such that $B(z,r)\subset\Omega$
all the partial derivatives $\frac{\partial f_i}{\partial y_j}$ belong to $L^1\big(\partial B(z,r),\cH^{n-1}\res\partial B(z,r)\big)$.
We choose one of these $r$, fix $B=B(z,r)$ and $B^{n-1}=\{x\in\R^{n-1}: |x|<1\}$. We take two bi-Lipschitz diffeomorphisms 
\[
\psi_i\colon B^{n-1}\to U_i\subset\partial B
\]
that define an orientation on $\partial B$, where
$i=1,2$ and $\partial B=U_1\cup U_2$. These conditions imply that the compositions $\frac{\partial f_L}{\partial y_I}\circ \psi_i$, with $I\in I_{n-1,n}$ and $L\in I_{n-1,m}$, are $\mathcal L^{n-1}$-measurable. We fix a standard partition of unity $\{\Upsilon_1,\Upsilon_2\}$ subordinate to the open covering $\{U_1,U_2\}$
and define for an $\mathcal H^{n-1}$--measurable $g\colon\partial B\to\bR$ the {\em oriented integral}
\beq\label{def-oriented}
\int_{\partial B}g\,df_L=\sum_{i=1}^2\int_{\psi_i^{-1}(U_i)}\!\!
\Upsilon_i(\psi_i(\xi))\,g\left(\psi_i(\xi)\right)\!\!\!\!\sum_{I\in I_{n-1,n}}\frac{\partial f_L}{\partial y_I}\left(\psi_i(\xi)\right)
\frac{\partial \left(\psi_i\right)_I}{\partial(x_1,\ldots,x_{n-1})}(\xi)\,d\xi,
\eeq
whenever it is well defined and independent from the choice of the partition of unity. From Cauchy-Schwarz inequality, this occurs for instance when  \[
\int_{\partial B}|g||Df|^{n-1}\,d\mathcal H^{n-1}<\infty,
\]
where we denote by $Df$ the Jacobian matrix of $f$ and with a slight abuse of notation we use the symbol $|\cdot|$ also for the Frobenius norm of matrices. Indeed, setting $L=(j_1,\ldots,j_{n-1})$, $F_L=(f_{j_1},\ldots,f_{j_{n-1}})$
and denoting by $JF_L$ the Jacobian of $F_L$, from the area formula and Hadamard's inequality we obtain
\begin{eqnarray}\label{estimation}
&&\sum_{i=1}^2\int_{\psi_i^{-1}(U_i)}\bigg|\Upsilon_i\circ \psi_i\;g\circ \psi_i\sum_{I\in I_{n-1,n}}\frac{\partial F_L}{\partial y_I}\circ \psi_i\;
\frac{\partial \left(\psi_i\right)_I}{\partial(x_1,\ldots,x_{n-1})}\bigg| \\
&&\le \sum_{i=1}^2\int_{\psi_i^{-1}(U_i)} \big|\Upsilon_i\circ \psi_i\;g\circ \psi_i\, JF_L\circ \psi_i\,J\psi_i\big| = \sum_{i=1}^2\int_{U_i} \big|\Upsilon_i\;g\; JF_L\big|\,d\cH^{n-1} \notag \\
&&=\int_{\partial B} |g|\, JF_L\,d\cH^{n-1}\leq \sqrt{n}\int_{\partial B} |g||\nabla f_{j_1}|\ldots|\nabla f_{j_{n-1}}|
\leq \sqrt{n}\int_{\partial B} |g|\, |Df|^{n-1}\,d\cH^{n-1}\,. \notag
\end{eqnarray}
Thus, it suffices to assume that $g\in L^\infty(\partial B,\cH^{n-1}\res\der B)$ and $|Df|\in L^{n-1}(\partial B,\cH^{n-1}\res\der B)$, which is the case for $\cH^1$-a.e. $r$, if $f\in W^{1,n-1}_\loc(\Omega;\bR^m)$.

The same assumptions apply in the case $n=2$, where we have oriented integration of Sobolev 1-forms over circles in $\bR^2$. If $f\in W^{1,1}_\loc(\Omega,\bR^m)$ for $m\geq1$ and $z\in\Omega$, then, as before, $\frac{\partial f_l}{\partial y_j}$ with $j=1,2$ and $l\in I_{1,m}$ belong to $L^1(\partial B,\cH^1\res \der B)$ for $\mathcal H^1$-almost every $r>0$ such that $B(z,r)\subset\Omega$. We fix one such $r$, set $B=B(z,r)$ and define
\[
v(t)=(v_1(t),v_2(t))=(z_1+r\cos t,z_2+r\sin t),
\]
where $t\in[-\pi,\pi]$. This curve parameterizes $\partial B$. For a measurable $g\colon\partial B\to\bR$, we set
$$
\int_{\partial B}g\,df_l=\int_{-\pi}^\pi g\circ v\sum_{j=1}^2\frac{\partial f_l}{\partial y_j}\circ v\;v_j',
$$
if the integral is defined. Since
$$
\int_{-\pi}^\pi\bigg|g\circ v\sum_{j=1}^2\frac{\partial f_l}{\partial y_j}\circ v\,v_j'\bigg|\leq\int_{-\pi}^\pi|g|\circ v\,|\nabla f_l|\circ v\,\left|v'\right|=\int_{\partial B}|g||\nabla f_l|d\mathcal H^1,
$$
as in the case $n>2$, the assumption $g\in L^\infty(\partial B,\cH^1\res\der B)$ suffices, because we have $|\nabla f_l|\in L^1(\partial B,\cH^1\res\der B)$ for every $l=1,\ldots,m$.


\section{Exterior differentiation of pullback Sobolev $(n-1)$-forms in $\R^n$}

\begin{lemma}\label{lemma}
Consider the integers $n\geq3$, $m\geq n-1$ and $J=(i_1,\ldots,i_{n-1})\in I_{n-1,m}$. Let $f,h\in W^{1,1}(\mathbb B,\bR^m)$ and let $r\in(0,1)$ be such that 
all $\ds\frac{\der f_i}{\der y_j}$ and $\ds\frac{\der h_i}{\der y_j}$ are measurable on $\der B(0,r)$ and we have
\[
\int_{\partial B(0,r)}|Df|^{n-1}d\mathcal H^{n-1}<\infty\quad \mbox{and } \quad \int_{\partial B(0,r)}|Dh|^{n-1}d\mathcal H^{n-1}<\infty.
\]
Then for every $g\in L^\infty\lls \der B(0,r),\cH^{n-1}\res\der B(0,r)\rrs$, the following estimate holds
\begin{align*}
&\left|\int_{\partial B(0,r)}g\,df_J-\int_{\partial B(0,r)}g\,dh_J\right| \leq C\, \|g\|_{L^\infty(\partial B(0,r))}\sum_{k=1}^{n-1}\left(\int_{\partial B(0,r)}|Df|^{n-1}d\mathcal H^{n-1}\right)^{\frac{k-1}{n-1}}\\
&\qquad\qquad\left(\int_{\partial B(0,r)}|Df-Dh|^{n-1}d\mathcal H^{n-1}\right)^{\frac{1}{n-1}} \left(\int_{\partial B(0,r)}|Dh|^{n-1}d\mathcal H^{n-1}\right)^{\frac{n-1-k}{n-1}}.
\end{align*}
\end{lemma}
\begin{proof}
We set the difference to be estimated
\[
\lambda=\left|\int_{\partial B(0,r)}g\,df_J-\int_{\partial B(0,r)}g\,dh_J\right|\,.
\]
Taking into account the definition \eqref{def-oriented} and the estimates \eqref{estimation}, we get
{\small
\begin{align*}
\lambda &= \left|\int_{\partial B(0,r)}g\,df_{i_1}\wedge\ldots\wedge df_{i_{n-1}}-\int_{\partial B(0,r)}g\,dh_{i_1}\wedge\ldots\wedge dh_{i_{n-1}}\right|\\
&\leq\sum_{k=1}^{n-1}\left|\int_{\partial B(0,r)}g\,df_{i_1}\wedge\ldots\wedge df_{i_{k-1}}\wedge\left(df_{i_k}-dh_{i_k}\right)\wedge dh_{i_{k+1}}\wedge\ldots\wedge dh_{i_{n-1}}\right|\\&\leq
C\sum_{k=1}^{n-1}\int_{\partial B(0,r)}|g||Df|^{k-1}|D(f-h)||Dh|^{n-1-k}d\mathcal H^{n-1}\\&\leq
C\,\|g\|_{L^\infty(\partial B(0,r))}\sum_{k=1}^{n-1}\left(\int_{\partial B(0,r)}|Df|^{n-1}d\mathcal H^{n-1}\right)^{\frac{k-1}{n-1}}
\left(\int_{\partial B(0,r)}|Df-Dh|^{n-1}d\mathcal H^{n-1}\right)^{\frac{1}{n-1}} \\
&\quad\left(\int_{\partial B(0,r)}|Dh|^{n-1}d\mathcal H^{n-1}\right)^{\frac{n-1-k}{n-1}}.
\end{align*}
}
\end{proof}

The following theorem corresponds to Theorem~\ref{th1} in the case $k=n-1$.

\begin{theorem}\label{k=n-1}
Let $m\geq n\geq2$ be positive integers and assume that one of the following conditions holds: $p>n-1$ if $n>2$, or $p=1$ if $n=2$. Let $\Omega\subset\R^n$ be open, let $f\in W^{1,p}_\loc(\Omega;\bR^m)$ and let $\eta$ be a continuously differentiable $(n-1)$-form of $\bR^m$. Then the condition $f^*\eta=0$ a.e.\ in $\Omega$ implies that $f^*(d\eta)=0$ a.e.\ in $\Omega$.
\end{theorem}
\begin{proof}
We fix $\eta =\ds\sum_{J\in I_{n-1,m}}\eta_J\,dx_J$, hence our assumption yields
$$
0=f^*\eta=\sum_{J\in I_{n-1,m}}\eta_J\circ f\,df_J\qquad \mbox{a.e.\ in $\Omega$.}
$$
Taking into account formula
$$
d\eta=\sum_{J\in I_{n-1,m}}d\eta_J\wedge dx_J=\sum_{J\in I_{n-1,m}}\sum_{j=1}^m\frac{\partial\eta_J}{\partial x_j}dx_j\wedge dx_J,
$$
our objective is to show that
\beq\label{eq:pullb}
f^*(d\eta)=\sum_{J\in I_{n-1,m}}\sum_{j=1}^m\frac{\partial\eta_J}{\partial x_j}\circ f\,df_j\wedge df_J=0\qquad \mbox{a.e.\ in $\Omega$.}
\eeq
Notice that the coefficients of the differential form in \eqref{eq:pullb} are well defined a.e., but they may not be locally integrable.
Let us consider the case $n>2$. For any mapping $\Phi\colon\Omega\to\bR^k$ with some $k\in\bN$, and $z\in\Omega$ and $r>0$ such that $B(z,r)\subset\Omega$, we denote
\begin{equation*}
\Phi^{z,r}(y)=\frac{\Phi(z+ry)-\Phi(z)}{r}.
\end{equation*}
Clearly, if $\Phi\in W^{1,p}(\Omega;\bR^k)$, then $\Phi^{z,r}\in W^{1,p}(\mathbb B;\bR^k)$ for $r>0$ sufficiently small.

Let us fix a point $z\in\Omega$, which is a Lebesgue point for both
\[
y\to |f(y)-f(z)|^p\quad\mbox{ and }\quad y\to |Df(y)-Df(z)|^p
\]
and such that the following $L^p$-differentiability holds
\begin{equation}\label{chudoshodimost}
\frac1{r^p}\meanint_{B(z,r)}|f(y)-f(z)-df_{z}(y-z)|^pdy\to0.
\end{equation}
It is well known that the set of all points with these properties
has full measure in $\Omega$, see for instance \cite{EvansGariepy}. We introduce the linear mapping $g\colon\R^n\to\bR^m$ defined by
\[
y\to g(y)=df_z(y).
\]
The limit
\[
\meanint_{\mathbb B}|f(z+ry)-f(z)|^pdy=\meanint_{B(z,r)}|f-f(z)|^p\to0 \quad \mbox{as $r\to0$}
\]
joined with the coarea formula provides an infinitesimal sequence
$\{r_i\}_{i\in\bN}\subset(0,1)$ such that for $\mathcal H^1$--a.e. $t\in(0,1)$ there holds
\beq\label{Prop1}
\lim_{i\to\infty}f(z+r_iy)=f(z) \quad \mbox{ for  $\mathcal H^{n-1}\text{--a.e.\ } y\in\partial B(0,t)$.}
\eeq
Since $p>n-1$, up to a modification of $f$ on an $\mathcal L^n$--negligible set we can find a set $S\subset(0,\max_ir_i)$ of full
measure such that 
\[
f|_{\partial B(z,t)}\in W^{1,p}\big(\partial B(z,t)\big)\cap C^0\big(\partial B(z,t)\big)
\]
for each $t\in S$. Then the set $\bigcap_{i\ge 1} \frac{S}{r_i}$ has also full measure in $(0,1)$ and for every $t$ in this set we have that
\begin{equation}
f^{z,r_i} \in W^{1,p}\big(\partial B(0,t)\big)\cap C^0\big(\partial B(0,t)\big)
\end{equation}
for each $i\in\bN$. Thus, for these $t$'s we can apply the Sobolev imbedding theorem on spheres, getting
\begin{align}\label{Prop2}
\sup_{v,w\in\partial B(0,t)}|f(z+r_iw)-f(z+r_iv)|=&\;r_i\sup_{v,w\in\partial B(0,t)}|f^{z,r_i}(w)-f^{z,r_i}(v)|\\
\leq&\; C\,t^{1-\frac{n-1}p}r_i\left(\int_{\partial B(0,t)}|Df^{z,r_i}|^pd\mathcal H^{n-1}\right)^{1/p}\notag
\end{align}
for each $i\in\bN$. 
From \eqref{chudoshodimost}, we have
\begin{equation*}
\int_{\mathbb B}|f^{z,r_i}-g|^p=\frac{C}{r_i^p}\,\meanint_{B(z,r_i)}|f(y)-f(z)-df_{z}(y-z)|^pdy\to0,
\end{equation*}
as $i\to\infty$. 
Again, up to extracting a subsequence we achieve
\begin{equation}\label{Prop4}
\lim_{i\to\infty}\int_{\partial B(0,t)}|f^{z,r_i}-g|^pd\mathcal H^{n-1}=0\quad\mbox{for $\mathcal H^1$--a.e.\ $t\in(0,1)$.}
\end{equation}
Similarly, since $z$ is also a Lebesgue point of $y\to |Df(y)-Df(z)|^p$ and
\beq\label{eq:fzri}
\left(df^{z,r_i}\right)_y=\left(df\right)_{z+r_iy},
\eeq
up to extracting a subsequence from $\{r_i\}_{i\in\N}$, we also obtain
\begin{equation}\label{Prop5}
\lim_{i\to\infty}\int_{\partial B(0,t)}|Df^{z,r_i}-Dg|^pd\mathcal H^{n-1}=0\quad \mbox{for $\mathcal H^1$--a.e.\ $t\in(0,1)$.}
\end{equation}
Repeating the application of the Sobolev imbedding theorem on spheres, we get
\begin{equation}\label{Prop6}
\sup_{v,w\in\partial B(0,t)}\!\!\!|(f^{z,r_i}-g)(w)-(f^{z,r_i}-g)(v)|\leq Ct^{1-\frac{n-1}p}\!\!
\left(\int_{\partial B(0,t)}\!\!\!\!|Df^{z,r_i}-Dg|^pd\mathcal H^{n-1}\right)^{1/p}
\end{equation}
for every $i\in\bN$, for $\mathcal H^1$--a.e.\ $t\in(0,1)$.
%
Joining a smoothing argument with Lemma~\ref{lemma}, we also have
\begin{equation}\label{Prop8}
\int_{\partial B(0,t)}df^{z,r_i}_J=0\,\,\text{ for each }\,\,i\in\bN\,\,\text{ and }\,\,J\in I_{n-1,m}
\end{equation}
for $\mathcal H^1$--almost every $t\in[0,1]$, since this property holds for smooth functions. 
%
%
Now, we apply our assumption $\sum_{J\in I_{n-1,m}}\eta_J\circ fdf_J=0$ and \eqref{eq:fzri} to obtain that for almost every $y\in\mathbb B$ there holds
\begin{equation*}
\sum_{J\in I_{n-1,m}}\eta_J\left(f(z+r_iy)\right)\left(df_J^{z,r_i}\right)_y=0
\end{equation*}
for each $i\in\bN$. Using Fubini theorem and adding the suitable terms to each side of the equality, we observe
\begin{align}\label{Prop9}
\sum_{J\in I_{n-1,m}}\left(\eta_J\circ f\right)^{z,r_i}(y)\left(df_J^{z,r_i}\right)_y=
-\frac{1}{r_i}\sum_{J\in I_{n-1,m}}\eta_J\left(f(z)\right)\left(df_J^{z,r_i}\right)_y
\end{align}
for $\mathcal H^{n-1}$--a.e. $y\in\partial B(0,t)$ for each $i\in\bN$ for $\mathcal H^1$--a.e. $t\in(0,1)$.

We can fix some $\tau_0\in(0,1)$ such that all properties from \eqref{Prop1} to \eqref{Prop9} hold with $t=\tau_0$. Our next objective is to prove the convergence
\begin{equation}\label{convergence}
\int_{\partial B(0,\tau_0)}\sum_{J\in I_{n-1,m}}\left(\eta_J\circ f\right)^{z,r_i}df_J^{z,r_i}\to
\int_{\partial B(0,\tau_0)}\sum_{J\in I_{n-1,m}}\sum_{j=1}^m\frac{\partial\eta_J}{\partial x_j}\left(f(z)\right)g_jdg_J,
\end{equation}
when $i\to\infty$. Notice that by the discussion in Section~\ref{OInt}, this sequence of oriented integrals is well defined for the choice of $\tau_0$, because the function $\left(\eta_J\circ f\right)^{z,r_i}$ is continuous on $\partial B(0,\tau_0)$ and
\[
\int_{\partial B(0,\tau_0)}|Df^{z,r_i}|^pd\mathcal H^{n-1}<\infty,
\]
due to \eqref{Prop5}. Next, we will establish the following
\begin{center}
\textbf{claim:} the sequence $\left\{\left(\eta_J\circ f\right)^{z,r_i}\right\}_i$ converges to $\ds \sum_{j=1}^m\frac{\partial\eta_J}{\partial x_j}\left(f(z)\right)g_j$ uniformly on $\partial B(0,\tau_0)$ for all $J\in I_{n-1,m}$.
\end{center}
Let us fix $J\in I_{n-1,m}$. By the property~\eqref{Prop4}, passing to another subsequence, we have $f^{z,r_i}(w)\to g(w)$ for $\mathcal H^{n-1}$--almost every $w\in\partial B(0,\tau_0)$. We fix one such $w$ and obtain by properties~\eqref{Prop6} and~\eqref{Prop5} that
\begin{equation*}
\sup_{y\in\partial B(0,\tau_0)}|(f^{z,r_i}-g)(y)-(f^{z,r_i}-g)(w)|\to0,
\end{equation*}
when $i\to\infty$. Hence the triangle inequality gives
\begin{equation}\label{pervajashodimost}
\lim_{i\to\infty}\sup_{y\in\partial B(0,\tau_0)}|f^{z,r_i}(y)-g(y)|=0.
\end{equation}
On the other hand, we similarly have convergence almost everywhere $f(z+r_iw)\to f(z)$ by~\eqref{Prop1} and the convergence
\begin{equation*}
\sup_{y,w\in\partial B(0,\tau_0)}|f(z+r_iy)-f(z+r_iw)|
\leq C\tau_0^{1-\frac{n-1}p}r_i\left(\int_{\partial B(0,\tau_0)}|Df^{z,r_i}|^pd\mathcal H^{n-1}\right)^{1/p}\to0,
\end{equation*}
when $i\to0$, implied by~\eqref{Prop2} and the fact that, by~\eqref{Prop5}, the sequence $$\left\{\int_{\partial B(0,\tau_0)}|Df^{z,r_i}|^p\right\}_i$$ is bounded from above independently of $i$. As before, we combine these two convergences to obtain $f(z+r_iy)\to f(z)$, when $i\to\infty$, uniformly in $y\in\partial B(0,t)$.

We continue applying the mean value theorem to the function $\eta_J$, obtaining for each $i\in\bN$ and $y\in\partial B(0,\tau_0)$ a point $\tau_{i,J,y}\in\bR^m$ belonging to the segment $[f(z),f(z+r_iy)]$, such that
\begin{eqnarray}\label{meanvalue}
\qquad\left(\eta_J\circ f\right)^{z,r_i}(y)
&=&\frac{1}{r_i}\bigl(\eta_J\left(f(z+r_iy)\right)-\eta_J\left(f(z)\right)\bigr) \\
& =&\frac{1}{r_i}\sum_{j=1}^m\frac{\partial\eta_J}{\partial x_j}(\tau_{i,J,y})\left(f_j(z+r_iy)-f_j(z)\right)
=\sum_{j=1}^m\frac{\partial\eta_J}{\partial x_j}(\tau_{i,J,y})f_j^{z,r_i}(y). \notag
\end{eqnarray}
In addition, since 
$$
\sup_{y\in\partial B(0,\tau_0)}|\tau_{i,J,y}-f(z)|\leq\sup_{y\in\partial B(0,\tau_0)}|f(z+r_iy)-f(z)|\to0,
$$
when $i\to\infty$, the continuity of $\ds\frac{\partial\eta_J}{\partial x_j}$ at the point $f(z)$ implies that there exists
\begin{equation}\label{vtorajashodimost}
\lim_{i\to\infty}\left(\max_{1\le j\le m} \sup_{y\in\partial B(0,\tau_0)}\left|\frac{\partial\eta_J}{\partial x_j}(\tau_{i,J,y})- \frac{\partial\eta_J}{\partial x_j}(f(z))\right|\right)=0
\end{equation}
Finally, it is rather easy to conclude from~\eqref{meanvalue},~\eqref{pervajashodimost} and~\eqref{vtorajashodimost} the following uniform convergence
$$
\left(\eta_J\circ f\right)^{z,r_i}(y)=\sum_{j=1}^m\frac{\partial\eta_J}{\partial x_j}(\tau_{i,J,y})f_{j}^{z,r_i}(y)
\to \sum_{j=1}^m\frac{\partial\eta_J}{\partial x_j}\left(f(z)\right)g_j(y),
$$
when $y$ varies in $\partial B(0,\tau_0)$. This completes the proof of the claim.

\noindent
Now, we estimate
\begin{align}\label{final}
&\left|\int_{\partial B(0,\tau_0)}\sum_{J\in I_{n-1,m}}\left(\eta_J\circ f\right)^{z,r_i}df_J^{z,r_i}-
\int_{\partial B(0,\tau_0)}\sum_{J\in I_{n-1,m}}\sum_{j=1}^m\frac{\partial\eta_J}{\partial x_j}\left(f(z)\right)g_jdg_J\right|\\
\leq&\sum_{J\in I_{n-1,m}}\left|\int_{\partial B(0,\tau_0)}\left[\left(\eta_J\circ f\right)^{z,r_i}-
\sum_{j=1}^m\frac{\partial\eta_J}{\partial x_j}\left(f(z)\right)g_j\right]df_J^{z,r_i}\right|\notag\\
+&\sum_{J\in I_{n-1,m}}\left|\int_{\partial B(0,\tau_0)}\left[\sum_{j=1}^m\frac{\partial\eta_J}{\partial x_j}\left(f(z)\right)g_j\right]
\left(df_J^{z,r_i}-dg_J\right)\right|.\notag
\end{align}
Taking into account \eqref{estimation}, we estimate the first sum on the right-hand side as follows
\begin{multline}\label{final1}
\sum_{J\in I_{n-1,m}}\left|\int_{\partial B(0,\tau_0)}\left[\left(\eta_J\circ f\right)^{z,r_i}-
\sum_{j=1}^m\frac{\partial\eta_J}{\partial x_j}\left(f(z)\right)g_j\right]df_J^{z,r_i}\right|\\\leq C\max_{J\in I_{n-1,m}}
\sup_{y\in\partial B(0,\tau_0)}\left|\left(\eta_J\circ f\right)^{z,r_i}(y)-
\sum_{j=1}^m\frac{\partial\eta_J}{\partial x_j}\left(f(z)\right)g_j(y)\right|\int_{\partial B(0,\tau_0)}|Df^{z,r_i}|^{n-1}d\mathcal H^{n-1}.
\end{multline}
Since the combination of Jensen's inequality and~\eqref{Prop5} implies
\begin{equation}\label{star}
\int_{\partial B(0,\tau_0)}|Df^{z,r_i}-Dg|^{n-1}d\mathcal H^{n-1}\to0\,\,\text{ when }\,\,i\to\infty,
\end{equation}
the factor $\int_{\partial B(0,\tau_0)}|Df^{z,r_i}|^{n-1}d\mathcal H^{n-1}$ on the right-hand side of~\eqref{final1} is bounded independently of $i$. Therefore, in view of the previous claim the sum in~\eqref{final1} converges to zero, as $i\to\infty$.
The convergence to zero of the second term on the right hand side of~\eqref{final} follows from Lemma~\ref{lemma},~\eqref{star} and again the boundedness of the sequence
\[
\int_{\partial B(0,\tau_0)}|Df^{z,r_i}|^{n-1}d\mathcal H^{n-1}.
\]
Thus, we conclude the validity of the limit \eqref{convergence}.
On the other hand, \eqref{Prop9} and~\eqref{Prop8} show that
\begin{align*}
\int_{\partial B(0,\tau_0)}\sum_{J\in I_{n-1,m}}\left(\eta_J\circ f\right)^{z,r_i}df_J^{z,r_i}
=-\frac{1}{r_i}\sum_{J\in I_{n-1,m}}\eta_J\left(f(z)\right)\int_{\partial B(0,\tau_0)}df_J^{z,r_i}=0,
\end{align*}
which together with~\eqref{convergence} implies
\begin{align*}
\int_{\partial B(0,\tau_0)}\sum_{J\in I_{n-1,m}}\sum_{j=1}^m\frac{\partial\eta_J}{\partial x_j}\left(f(z)\right)g_jdg_J=0.
\end{align*}
We apply the Stokes theorem to the last integral, obtaining
\begin{eqnarray*}
0&=&\int_{B(0,\tau_0)}\sum_{J\in I_{n-1,m}}\sum_{j=1}^m\frac{\partial\eta_J}{\partial x_j}\left(f(z)\right)dg_j\wedge dg_J \\
&=&
\mathcal L^n(B(0,\tau_0))\sum_{J\in I_{n-1,m}}\sum_{j=1}^m\frac{\partial\eta_J}{\partial x_j}\left(f(z)\right)\frac{\partial (f_j,f_J)}{\partial(y_1,\ldots,y_n)}(z),
\end{eqnarray*}
which yields
\begin{equation*}
\sum_{J\in I_{n-1,m}}\sum_{j=1}^m\frac{\partial\eta_J}{\partial x_j}\left(f(z)\right)\left(df_j\wedge df_J\right)_z=0.
\end{equation*}
The proof of the theorem in the case $n>2$ is complete, since the set of points $z$ with the required properties has full measure in $\Omega$.

The proof in the case $n=2$ is simpler. The details are left to the reader. The major difference compared to the previous proof is that the Sobolev imbedding theorem on circles for a mapping $f\in W^{1,1}(\Omega,\bR^m)$, $\Omega\subset\bR^2$, has the form
$$
\sup_{y,v\in\partial B(z,r)}|f(y)-f(v)|\leq C\int_{\partial B(z,r)}|Df|d\mathcal H^1
$$
for $z\in\Omega$ and $\mathcal H^1$-almost every $r>0$, such that $B(z,r)\subset\Omega$.
\end{proof}

\section{Slicing and lower dimensional pullback Sobolev differential forms}\label{Sect:Slice}

In this section we complete the proof of Theorem~\ref{th1}, considering the case $k<n-1$.
We will follow the slicing argument of \cite{MMM}, considering $(k+1)$-dimensional sections of the space $\bR^n$. Let us introduce the notation we need for this purpose. 

We write $(e_1,\ldots,e_n)$ for the canonical basis of $\bR^n$ and for a nonempty set of indices $\Gamma\subset\{1,\ldots,n\}$, we define the projections
\[
\pi_\Gamma\colon\bR^n\to\spa\{e_j\colon j\in\Gamma\}\quad \mbox{and}\quad \hat{\pi}_\Gamma\colon\bR^n\to \spa\{e_j\colon j\in\Gamma\}^\bot
\]
by $\pi_\Gamma(x)=\sum_{j\in\Gamma}x_je_j$ and $\hat{\pi}_\Gamma(x)=x-\pi_\Gamma(x)$ for each $x=(x_1,\ldots,x_n)\in\bR^n$. If $Q$ is an open $n$-dimensional interval in $\bR^n$, namely the product of $n$ open intervals, once a nonempty subset $\Gamma\subsetneq\{1,\ldots,n\}$ is fixed, we denote $Q_\Gamma=\pi_\Gamma(Q)$ and $\hat{Q}_\Gamma=\hat{\pi}_\Gamma(Q)$. Finally, for a function $u\colon Q\to\bR$ and a point $z\in\hat{Q}_\Gamma$, the section $u^\zr\colon Q_\Gamma\to\bR$ is given by $u^\zr(y)=u(z+y)$ for each $y\in Q_\Gamma$. We utilize the following fact about Sobolev mappings (see, for instance, \cite[Proposition~2.2]{MMM}).
\begin{lemma}\label{slices}
Let $n\geq2$ be an integer, $p\geq1$, $\emptyset\neq\Gamma\subsetneq\{1,\ldots,n\}$ and $Q$ be an open $n$-dimensional interval. Assume $u\in W^{1,p}(Q)$. Then for almost every $z\in\hat{Q}_\Gamma$, we have $u^\zr\in W^{1,p}(Q_\Gamma)$ and $\partial_ku^\zr=(\partial_ku)^\zr$ almost everywhere in $Q_\Gamma$, where $k\in\Gamma$.
\end{lemma}
\begin{proof}[Proof of Theorem~\ref{th1}]
By the assumption, we have
\begin{equation}\label{assumption}
0=f^*\eta=\sum_{J\in I_{k,m}}\eta_J\circ fdf_J
\end{equation}
almost everywhere in $\Omega$. Due to Theorem~\ref{k=n-1}, we are left with the case $k<n-1$. Without the loss of generality we may assume that $\Omega=Q$ is an open interval. Let us fix $I=(i_1,\ldots,i_{k+1})\in I_{k+1,n}$ and $\Gamma=\{i_1,\ldots,i_{k+1}\}$. Using Fubini's theorem and~\eqref{assumption}, we deduce that for $\mathcal H^{n-k-1}$-almost every $z\in\hat{Q}_\Gamma$ and for each $q\in\{1,\ldots,k+1\}$, denoting $I_q=(i_1,\ldots,\widehat{i_q},\ldots,i_{k+1})\in I_{k,n}$, we have
$$
\sum_{J\in I_{k,m}}\left(\eta_J\circ f\right)^\zr\left(\frac{\partial f_J}{\partial y_{I_q}}\right)^\zr=0
$$
almost everywhere in $Q_\Gamma$. Then by Lemma~\ref{slices}, we observe
$$
\sum_{J\in I_{k,m}}\eta_J\circ f^\zr\frac{\partial f^\zr_J}{\partial y_{I_q}}=0
$$
for each $q\in\{1,\ldots,k+1\}$ almost everywhere in $Q_\Gamma$ for $\mathcal H^{n-k-1}$-almost every $z\in\hat{Q}_\Gamma$; and $f^\zr\in W^{1,p}(Q_\Gamma;\bR^m)$ for these $z$. That is,
$$
0=\sum_{J\in I_{k,m}}\eta_J\circ f^\zr df_J^\zr=\left(f^\zr\right)^*\eta
$$
almost everywhere in $Q_\Gamma$ for these $z$. Applying Theorem~\ref{k=n-1} for these $z$ gives 
$$
\left(f^\zr\right)^*(d\eta)=\sum_{J\in I_{k,m}}\sum_{j=1}^m\frac{\partial\eta_J}{\partial x_j}\circ f^\zr df^\zr_j\wedge df_J^\zr=0
$$
almost everywhere in $Q_\Gamma$. Thus, we have
\begin{equation*}
0=\sum_{J\in I_{k,m}}\sum_{j=1}^m\frac{\partial\eta_J}{\partial x_j}\circ f^\zr\frac{\partial(f_j^\zr,f^\zr_J)}{\partial y_I}
=\sum_{J\in I_{k,m}}\sum_{j=1}^m\left(\frac{\partial\eta_J}{\partial x_j}\circ f\,\frac{\partial(f_j,f_J)}{\partial y_I}\right)^\zr
\end{equation*}
almost everywhere in $Q_\Gamma$ for $\mathcal H^{n-k-1}$-almost every $z\in\hat{Q}_\Gamma$. Fubini theorem implies
$$
\sum_{J\in I_{k,m}}\sum_{j=1}^m\frac{\partial\eta_J}{\partial x_j}\circ f\,\frac{\partial(f_j,f_J)}{\partial y_I}=0
$$
almost everywhere in $Q$. Finally, the arbitrariness of the choice of $I$ yields
$$
0=\sum_{J\in I_{k,m}}\sum_{j=1}^m\frac{\partial\eta_J}{\partial x_j}\circ f\,df_j\wedge df_J=f^*(d\eta)
\quad \mbox{ a.e.\ in $Q$.}
$$
This concludes the proof.
\end{proof}

\section{Hypersurfaces in stratified nilpotent Lie groups}\label{HypSect}

In this section we give the proof of Theorem~\ref{main}. The following algebraic lemma will play an important role.

\begin{lemma}\label{algebra}
Let $\xi_1,\ldots,\xi_{q-1}\in V$ be vectors in a $q$-dimensional linear space $V$ with $q\geq3$. Assume that $X_1,\ldots,X_{q}$ form the basis of $V$ and that $\eta_1,\ldots,\eta_q$ is its dual basis of $V^*$. Then $X_s\in\Xi:=\spa\{\xi_1,\ldots,\xi_{q-1}\}$ implies $\eta_1\wedge\ldots\wedge\widehat{\eta}_s\wedge\ldots\wedge\eta_q(\xi_1\wedge\ldots\wedge\xi_{q-1})=0$ for $s\in\{1,\ldots,q\}$. Conversely, if the vectors $\xi_1,\ldots,\xi_{q-1}$ are linearly independent, we have $X_s\in\Xi$, whenever $\eta_1\wedge\ldots\wedge\widehat{\eta}_s\wedge\ldots\wedge\eta_q(\xi_1\wedge\ldots\wedge\xi_{q-1})=0$.
\end{lemma}
\begin{proof}
If $X_s\in\Xi$ and the vectors $\xi_1,\ldots,\xi_{q-1}$ are linearly dependent, the claim is trivial, because $\xi_1\wedge\ldots\wedge\xi_{q-1}=0$. Otherwise, we have $\xi_1\wedge\ldots\wedge\xi_{q-1}=\alpha X_s\wedge Z_1\wedge\ldots\wedge Z_{q-2}$, where $\alpha\in\bR$ and $Z_1,\ldots,Z_{q-2}\in V$ are such that the vectors $X_s,Z_1,\ldots,Z_{q-2}$ form the basis of $\Xi$. As a result, we have
$$
\eta_{1}\wedge\ldots\wedge\widehat{\eta}_s\wedge\ldots\wedge\eta_{q}\left(X_s\wedge Z_1\wedge\ldots\wedge Z_{q-2}\right)=0,
$$
proving the direct implication.

In order to prove the other implication, we consider the expansion of the vectors $\xi_1,\ldots,\xi_{q-1}$ in terms of the basis $X_1,\ldots,X_{q}$:
\begin{equation*}
\xi_j=\sum_{i=1}^q\alpha_i^jX_i
\end{equation*}
for $j\in\{1,\ldots,q-1\}$ and some $\alpha_i^j\in\bR$. Then
\begin{align*}
\eta_1\wedge\ldots\wedge\widehat{\eta}_s\wedge&\ldots\wedge\eta_q(\xi_1\wedge\ldots\wedge\xi_{q-1})\\=&
\eta_1\wedge\ldots\wedge\widehat{\eta}_s\wedge\ldots\wedge\eta_q
\left(\sum_{r=1}^q\det\left[\alpha^j_{i}\right]_{j=1,\ldots,q-1}^{i=1,\ldots,\hat{r},\ldots,q}
X_{1}\wedge\ldots\wedge\widehat{X}_r\wedge\ldots\wedge X_{q}\right)\\=&\det\left[\alpha^j_{i}\right]_{j=1,\ldots,q-1}^{i=1,\ldots,\hat{s},\ldots,q}
\end{align*}
by the duality. Thus if $\eta_1\wedge\ldots\wedge\widehat{\eta}_s\wedge\ldots\wedge\eta_q(\xi_1\wedge\ldots\wedge\xi_{q-1})=0$, the determinant on the right-hand side of the last equation is equal to zero. Hence, one of the columns of the matrix is a linear combination of the other $q-2$. More precisely, there exist $k\in\{1,\ldots,q\}\setminus\{s\}$ and $b_i\in\bR$ with $i\in\{1,\ldots,q\}\setminus\{s,k\}$, such that $$\alpha^j_k=\sum_{\substack{i=1\\i\neq k,s}}^qb_i\alpha^j_i$$ for each $j\in\{1,\ldots,q-1\}$. Therefore,
\begin{align*}
\xi_j=\alpha^j_sX_s+\left(\sum_{\substack{i=1\\i\neq k,s}}^qb_i\alpha^j_i\right)X_k+\sum_{\substack{i=1\\i\neq k,s}}^q\alpha^j_iX_i
=\alpha^j_sX_s+\sum_{\substack{i=1\\i\neq k,s}}^q\alpha^j_i\left(X_i+b_iX_k\right),
\end{align*}
which yields $\Xi\subset\spa\left\{X_s,X_i+b_iX_k\colon i\in\{1,\ldots,q\}\setminus\{k,s\}\right\}$. Thus, the assumption $\dim\Xi=q-1$ implies that $X_s\in\Xi$.
\end{proof}

We introduce a few basic facts concerning stratified nilpotent Lie groups.
A stratified nilpotent Lie group $\G$ can be seen as a linear space equipped with a Lie group operation and a corresponding Lie algebra $\cg=V_1\oplus\cdots\oplus V_\iota$, where $\iota$ is the step of the group and the conditions
\[
[V_1,V_j]=V_{j+1} \quad \mbox{and} \quad [V_1,V_\iota]=\{0\}
\]
hold for each $j=1,\ldots,\iota-1$. We denote by $q$ the dimension of $\G$ as a linear space.
The Lie group operation is provided by the Baker-Campbell-Hausdorff formula, by equipping $\G$ also with the structure of Lie algebra.
Then the unit element is given by the origin of the linear space $\G$. The Lie algebra $\cg$ also defines a grading of on the group itself, setting
\[
H_j=\{X(0): X\in V_j\}\subset T_0\G
\]
the natural identification of $\G$ with $T_0\G$ gives
\[
\G=H_1\oplus H_2\oplus\cdots\oplus H_\iota.
\]
This grading of the space transfers to each point fiber of $T\G$ by left translations.
At each point $z\in\G$, for each $j=1,\ldots,\iota$, we set
\beq\label{eq:Hj}
(H_j)_z=(dl_z)_0(H_j),
\eeq
where $l_z:\G\to\G$ is the left translation $x\to zx$.
For the proof of Theorem~\ref{main}, it will be convenient to introduce the integers
$$
m_k=\sum_{j=1}^k\dim V_j
$$
for $k\in\{1,\ldots,\iota\}$. For each $i\in\{1,\ldots,q\}$, $d_i\in\{1,\ldots,\iota\}$ stands for its degree, the unique integer satisfying $m_{d_i-1}<i\leq m_{d_i}$, where we set $m_0=0$.

We fix a graded basis $X_1,\ldots,X_q$ of $\mathfrak{g}$, namely
\[
X_{m_k+1},X_{m_k+2},\ldots, X_{m_{k+1}}
\]
is a basis of $V_{k+1}$ for every $k=0,1,\ldots,\iota-1$, along with its dual basis of left invariant forms $\eta_1,\ldots,\eta_q$ in $\cg^*$.
This basis is characterized by the property
\[
\eta_r(X_s)=\delta_{r,s}=\left\{\begin{array}{ll} 1 & \mbox{if $r=s$} \\
0 & \mbox{otherwise}
\end{array}\right..
\]
The graded basis $\{X_i\}_{1\le i\le q}$ also allows us to identify $\G$ with $\R^q$ by defining $X_i(0)=e_i$ for all $i$ and setting the corresponding basis $(e_1,\ldots,e_q)$ of the linear space $\G$. This equips $\G$ with an auxiliary scalar product that makes the basis $(e_1,\ldots,e_n)$ orthonormal.

The Lie algebra structure yields the coefficients $c_{j,i}^k$ such that
\begin{equation}\label{comm}
[X_i,X_j]=\sum_{k=1}^q c_{i,j}^k X_k
\end{equation}
where $i,j,k\in\{1,\ldots,q\}$. Then the Maurer-Cartan equations are given by
$$
d\eta_k=\sum_{1\leq j<i\leq q}c_{j,i}^k\eta_i\wedge\eta_j
$$
for each $k=1,\ldots,q$, see for instance \cite{warner}.
The left invariance of the dual basis yields
\begin{equation}\label{struct}
d\eta_k(X_i\wedge X_j)=\eta_k([X_j,X_i])=c_{j,i}^k.
\end{equation}
On the other hand, since $[V_{d_i},V_{d_j}]\subset V_{d_i+d_j}$ for all $j,i\in\{1,\ldots,q\}$, \eqref{comm} implies $c_{i,j}^k=0$ whenever $d_i+d_j\neq d_k$, hence the Maurer-Cartan equations take the form
\begin{equation}\label{differential}
d\eta_k=\sum_{\substack{1\leq j<i\leq q\\d_i<d_k}}c_{j,i}^k\eta_i\wedge\eta_j
\quad \mbox{for each $k=1,\ldots,q$.}
\end{equation}
\begin{theorem}\label{main}
Let $\R^q=H_1\oplus\cdots\cdots\oplus H_\iota$ be equipped with the structure of noncommutative stratified group, with Lie algebra $\mathfrak{g}=V_1\oplus\ldots\oplus V_\iota$, for some $\iota\ge2$. Fix $m_1=\dim V_1$ and $p>q-m_1$ if $q>m_1+1$ or $p=q-m_1=1$ if $q=m_1+1$. Suppose that $f\in W_\loc^{1,p}(\Omega,\bR^q)$, where $\Omega\subset\bR^{q-1}$ is an open set. If $m_1<q-1$, we assume that
\begin{equation}\label{horiz}
(H_1)_{f(y)}\subset df_y(\bR^{q-1}_y)
\end{equation}
for almost every $y\in\Omega$; otherwise, we require
\begin{equation}\label{horiz1}
df_y(\bR^{q-1}_y)\subset (H_1)_{f(y)}
\end{equation}
for almost every $y\in\Omega$. Then there exists a set of positive measure $A\subset\Omega$, such that $\rk\, df<q-1$ in $A$. Moreover, if the group has step $2$, then the set $A$ can be chosen to have full measure in $\Omega$.
\end{theorem}
\begin{proof}
\textbf{Step 1.} ({\em Reformulating of the horizontality condition})

Assume first that $m_1<q-1$. Fix some $\kappa\in\{1,\ldots,\iota-1\}$. Suppose $y\in\Omega$ is such that $(H_1)_{f(y)},\ldots,(H_\kappa)_{f(y)}\subset df_y(\bR^{q-1}_y)$, which in the case $\kappa=1$ is equivalent to~\eqref{horiz}. We show that for this $y$ we have
\begin{equation}\label{ref1}
\left(f^*(\eta_{m_\kappa+1}\wedge\ldots\wedge\eta_q)\right)_y=0.
\end{equation}
Indeed, if necessary, we complete the collection $\{X_1(f(y)),\ldots,X_{m_\kappa}(f(y))\}$ with the vectors $Z^1_y,\ldots,Z^p_y\in\bR^q_{f(y)}$ to obtain the basis $\{X_1(f(y)),\ldots,X_{m_\kappa}(f(y)),Z^1_y,\ldots,Z^p_y\}$ of $df_y(\bR^{q-1}_y)$. The number $p$ is at most $q-1-m_\kappa\geq0$. Thus, any element of the corresponding basis of $\Lambda_{q-m_\kappa}(df_y(\bR^{q-1}_y))$ will contain at least one $X_i(f(y))$ for some $i\in\{1,\ldots,m_\kappa\}$. Therefore, $(\eta_{m_\kappa+1}\wedge\ldots\wedge\eta_q)_{f(y)}(\tau)=0$ for each $\tau\in\Lambda_{q-m_\kappa}(df_y(\bR^{q-1}_y))$ by the duality of $X_1,\ldots,X_{q}$ and $\eta_1,\ldots,\eta_q$. This implies~\eqref{ref1}.

In the case $m_1=q-1$ and $y\in\Omega$ satisfies~\eqref{horiz1}, the duality yields $(\eta_q)_{f(y)}(\tau)=0$ for each $\tau\in df_y(\bR^{q-1}_y)$. Thus, $(f^*\eta_q)_y=0$, which is equivalent to~\eqref{ref1}, when $\kappa=1$ and $m_\kappa+1=q$. 

\textbf{Step 2.} ({\em Induction step})

Let us assume that $f$ satisfies
\begin{equation}\label{ref2}
f^*(\eta_{m_\kappa+1}\wedge\ldots\wedge\eta_q)=0
\end{equation}
almost everywhere in $\Omega$ for some $\kappa\in\{1,\ldots,\iota-1\}$. In this step, we show that this assumption implies
\begin{equation*}
f^*(\eta_{1}\wedge\ldots\wedge\widehat{\eta}_s\wedge\ldots\wedge\eta_{q})=0
\end{equation*}
almost everywhere, for each $s$ with $d_s=\kappa+1$.

Let us fix one $s$ with $d_s=\kappa+1$. The condition $V_{\kappa+1}=[V_1,V_\kappa]$ means that $X_s\in V_{\kappa+1}$ may be represented as
\begin{align}\label{key}
X_s=\sum_{\substack{d_k=1,d_l=\kappa\\k<l}}\gamma_{k,l}^s[X_k,X_l]
\end{align}
for some $\gamma_{k,l}^s\in\bR$. Therefore, for each $r$ with $d_r>\kappa$, we have
\begin{equation}\label{delta}
\delta_{s,r}=\eta_r(X_s)=\sum_{\substack{d_k=1,d_l=\kappa\\k<l}}\gamma_{k,l}^s\eta_r\left([X_k,X_l]\right)
=\sum_{\substack{d_k=1,d_l=\kappa\\k<l}}\gamma_{k,l}^sc_{k,l}^r
\end{equation}
by~\eqref{struct}.
On the other hand, using~\eqref{differential}, we notice
\begin{align*}
d\eta_r\wedge\eta_{m_\kappa+1}\wedge\ldots\wedge\eta_{m_{d_r-1}}=&\sum_{\substack{1\leq j<i\leq q\\d_i<d_r}}c_{j,i}^r\eta_i\wedge\eta_j
\wedge\eta_{m_\kappa+1}\wedge\ldots\wedge\eta_{m_{d_r-1}}\\=&\sum_{\substack{1\leq j<i\leq m_\kappa}}c_{j,i}^r\eta_i\wedge\eta_j
\wedge\eta_{m_\kappa+1}\wedge\ldots\wedge\eta_{m_{d_r-1}}
\end{align*}
for each $r$ with $d_r>\kappa$ (if $d_r=\kappa+1$, the exterior product by $\eta_{m_\kappa+1}\wedge\ldots\wedge\eta_{m_{d_r-1}}$ is understood as the product by the scalar 1). We consider the form
$$
\theta_s=\sum_{\substack{d_k=1,d_l=\kappa\\k<l}}(-1)^{h(m_\kappa,k,l)}\gamma^s_{k,l}
\eta_1\wedge\ldots\wedge\widehat{\eta}_k\wedge\ldots\wedge\widehat{\eta}_l\wedge\dots\wedge\eta_{m_\kappa},
$$
where $\gamma^s_{k,l}$ are the coefficients from~\eqref{key} and the exponents $h(m_\kappa,k,l)\in\bN$ are chosen so that
$$
(-1)^{h(m_\kappa,k,l)}\eta_1\wedge\ldots\wedge\widehat{\eta}_k\wedge\ldots\wedge\widehat{\eta}_l\wedge\dots\wedge\eta_{m_\kappa}\wedge\eta_l\wedge\eta_k=
\eta_1\wedge\ldots\wedge\eta_{m_\kappa}.
$$
Notice that in the case $\kappa=1$ and $m_1=2$ the form $\theta_s=\theta_3$ becomes the scalar number $-\gamma^3_{1,2}$.
From the definition of $\theta_s$,  for each $r$ with $d_r>\kappa$ we observe that
\begin{align}\label{product}
\theta_s\wedge d\eta_r\wedge&\eta_{m_\kappa+1}\wedge\ldots\wedge\eta_{m_{d_r-1}}\\
=&\sum_{\substack{d_k=1,d_l=\kappa\\k<l}}(-1)^{h(m_\kappa,k,l)}\gamma^s_{k,l}c_{k,l}^r
\left(\bigwedge_{\substack{i=1\\i\neq k,l}}^{m_\kappa}\eta_i\right)\wedge\eta_l\wedge\eta_k
\wedge\eta_{m_\kappa+1}\wedge\ldots\wedge\eta_{m_{d_r-1}}\notag\\
=&\sum_{\substack{d_k=1,d_l=\kappa\\k<l}}\gamma^s_{k,l}c_{k,l}^r\eta_1\wedge\ldots\wedge\eta_{m_{d_r-1}}
=\delta_{r,s}\eta_1\wedge\ldots\wedge\eta_{m_{d_r-1}}\notag
\end{align}
by~\eqref{delta}.
In view of Theorem~\ref{th1}, we can differentiate \eqref{ref2}, obtaining
\begin{equation*}
0=f^*\left(d(\eta_{m_\kappa+1}\wedge\ldots\wedge\eta_q)\right)
=f^*\left(\sum_{r=m_\kappa+1}^{q}(-1)^{r-m_\kappa+1}d\eta_r\wedge(\eta_{m_\kappa+1}\wedge\ldots\wedge\widehat{\eta}_r\wedge\ldots\wedge\eta_q)\right)
\end{equation*}
almost everywhere in $\Omega$. Multiplication of this equation by $f^*\theta_s$ together with~\eqref{product} finally implies
\beqas
0&=&f^*\left(\sum_{r=m_\kappa+1}^{q}(-1)^{r-m_\kappa+1}\theta_s\wedge d\eta_r\wedge(\eta_{m_\kappa+1}\wedge\ldots\wedge\widehat{\eta}_r\wedge\ldots\wedge\eta_q)\right)\\
&=& (-1)^{s-m_\kappa+1}\,f^*\left(\eta_1\wedge\ldots\wedge\widehat{\eta}_s\wedge\ldots\wedge\eta_q\right)
\eeqas
almost everywhere in $\Omega$, where $d_s=\kappa+1$.

\textbf{Step 3.} ({\em Conclusion})

First, we assume that $\iota=2$. Our aim is to prove $df_y(\partial_1)\wedge\ldots\wedge df_y(\partial_{q-1})=0$ for almost every $y\in\Omega$. This equality is true for $y\in\Omega$, if $\theta(df_y(\partial_1)\wedge\ldots\wedge df_y(\partial_{q-1}))=0$ for each $\theta\in\Lambda_{q-1}\left(\bR_{f(y)}^q\right)^*$. Since
$$
\left\{(\eta_1\wedge\ldots\wedge\widehat{\eta}_s\wedge\ldots\wedge\eta_q)_{f(y)}\colon s=1,\ldots,q\right\}
$$
is the basis of $\Lambda_{q-1}\left(\bR_{f(y)}^q\right)^*$, it is enough to establish
\begin{equation}\label{goal1}
(\eta_{1}\wedge\ldots\wedge\widehat{\eta}_s\wedge\ldots\wedge\eta_{q})_{f(y)}\left(df_y(\partial_1)\wedge\ldots\wedge df_y(\partial_{q-1})\right)=0
\end{equation}
for each $s\in\{1,\ldots,q\}$ and almost every $y\in\Omega$. In the case $m_1=q-1$, all the vectors $df_y(\partial_j)$ are a.e.\ horizontal and this implies that \eqref{goal1} holds a.e.,\ for each $s$ with $d_s=1$. In the case $m_1<q-1$, by Lemma~\ref{algebra} the equality \eqref{goal1} holds a.e. for each $s$ with $d_s=1$. 
Furthermore, by step~1 applied to $\kappa=1$, in both of these cases 
\begin{equation}
f^*(\eta_{m_1+1}\wedge\ldots\wedge\eta_q)=0 \end{equation}
holds a.e., then step 2 implies that $f$ satisfies
$$
f^*(\eta_{1}\wedge\ldots\wedge\widehat{\eta}_s\wedge\ldots\wedge\eta_{q})=0
$$
almost everywhere in $\Omega$ for each $s$ with $d_s=2$. As a result, \eqref{goal1} holds for almost every $y\in\Omega$ and each $s=1,\ldots,q$, completing the proof for $\iota=2$.

If $\iota>2$, we necessarily have $m_1<q-1$. Then we argue by contradiction, assuming that $\rk df=q-1$ almost everywhere in $\Omega$. 
Since \eqref{horiz} holds a.e., by induction we can assume that
\[
(H_i)_{f(y)}\subset df_y(\bR^{q-1}_y)
\]
a.e.\ in $\Omega$, for each $i\leq\kappa$ and with $\kappa\in\{1,\ldots,\iota-1\}$.
The combination of step~1 and step~2 implies
$$
f^*(\eta_{1}\wedge\ldots\wedge\widehat{\eta}_s\wedge\ldots\wedge\eta_{q})=0
$$
a.e.\ in $\Omega$, for each $s$ with $d_s=\kappa+1$. This shows that \eqref{goal1} holds a.e.\ in $\Omega$, for each $s$ with $d_s=\kappa+1$. Our assumption on the rank of $df$ joined with Lemma~\ref{algebra} show that 
$X_s(f(y))\in df_y(\bR^{q-1}_y)$ for almost every $y\in\Omega$ and every $s$ such that $d_s=\kappa+1$. That is $(H_{\kappa+1})_{f(y)}\subset df_y(\bR^{q-1}_y)$ almost everywhere in $\Omega$, which gives the induction step.
As a consequence, we are lead to the inclusions $(H_i)_{f(y)}\subset df_y(\bR^{q-1}_y)$ that hold for almost every $y\in\Omega$ and each $i\in\{1,\ldots,\iota\}$, which is a contradiction. 
\end{proof}

\bibliography{bibtex}{}
\bibliographystyle{plain}

\end{document}